\documentclass[sn-mathphys-num]{sn-jnl}
\usepackage{graphicx}%
\usepackage{multirow}%
\usepackage{amsmath,amssymb,amsfonts}%
\usepackage{amsthm}%
\usepackage{mathrsfs}%
\usepackage[title]{appendix}%
\usepackage{xcolor}%
\usepackage{textcomp}%
\usepackage{manyfoot}%
\usepackage{booktabs}%
\usepackage{algorithm}%
\usepackage{algorithmicx}%
\usepackage{algpseudocode}%
\usepackage{listings}%
\usepackage{epstopdf}
\usepackage{subfigure}
\usepackage{caption}

\theoremstyle{thmstyleone}%
\newtheorem{theorem}{Theorem}
\newtheorem{lemma}[theorem]{Lemma}%

\theoremstyle{thmstyletwo}%
\newtheorem{example}{Example}%

\theoremstyle{thmstylethree}%

\begin{document}

\title{A novel  numerical method for mean field stochastic  differential equation}

\author[1]{Jinhui Zhou}\email{jhzhou21@mails.jlu.edu.cn}

\author[1]{Yongkui Zou}\email{zouyk@jlu.edu.cn}

\author*[1]{Shimin Chai}\email{chaism@jlu.edu.cn}

\author[1]{Boyu Wang}\email{wangby23@mails.jlu.edu.cn}

\author[1]{Ziyi Tan}\email{tanzy1021@mails.jlu.edu.cn}

\affil[1]{\orgdiv{School of Mathematics, Jilin University}, \orgaddress{ \city{Changchun}, \postcode{130012},  \country{China}}}

\abstract{In this paper, we propose a novel method to approximate the mean field stochastic  differential equation by means of  approximating the density function via Fokker-Planck equation.  We construct a well-posed truncated Fokker-Planck equation whose solution is an approximation to  the density function of solution to the mean field stochastic  differential equation. We also apply finite difference method to approximate the truncated Fokker-Planck equation and derive error estimates. We use the numerical density function  to replace the true measure in mean field stochastic  differential equation and set up a stochastic  differential equation to approximate the mean field one. Meanwhile, we derive  the corresponding error estimates. Finally, we present several numerical experiments to illustrate the theoretical analysis.}



\maketitle

\section{Introduction}
Stochastic differential equation (SDE for short) is a fundamental tool for modeling the evolution of various phenomena in different fields involving noise, cf.\ finance \cite{BlaSch1973}, chemical reaction \cite{San2014} and  environment science \cite{Dag1982}.
The mean field SDE, whose  coefficient functions depend on the law of the solution, plays an important role  in many fields:  the Hodgkin-Huxley model for neuron activation in neuroscience \cite{Bal2012},  the Patlak-Keller-Segel equations in biology and chemistry \cite{Cli1953,Eve1970} and  the Atlas models for equity markets in financial mathematics \cite{BanFerKar2005,BenJul2014}.

In recent years,  numerical approximations to the mean field SDE  have attracted many attentions.
Its main idea  is to apply an interactive particle system to approximate
the mean field SDE and then use an empirical measure to approach the joint distribution of particles. Therefore,  the interactive particle system is usually transformed to a high dimensional SDE and various effective  numerical methods can be applied to set up numerical schemes \cite{BosTal1997,dosEng2022,Chedos2022,ReiSto2022}.
Belomestny and Schoenmakers \cite{BelSch2018} proposed a projection-based particle method.
Liu \cite{Liu2024} proposed  three numerical schemes  and studied  the convergence rate for
the first one. Error estimates to the particle method  showed that enlarging  particle number improves the approximate accuracy. However, the simulation for empirical measure with a large number of particles is  extremely costly. To overcome this difficulty,  a random batch method was recently proposed and analyzed in \cite{JinLiLiu2020,JiLiLiu2021,JinLi2022,JinLi[2022]}, which was designed to parallel simulate several small particle systems at each iteration step. This method largely reduced the computational complexity  but raise the approximate error  if the  batch number is small.

In this paper, we will propose a novel method for a class of mean field SDEs based upon Fokker-Planck equation. In fact, the density function of the solution to a mean field SDE satisfies a nonlinear Fokker-Planck equation. Then we apply a finite difference method to the Fokker-Planck equation to obtain an approximate density function. Afterwards,  we use the approximate density function to replace the true density function  and transform  the mean field SDE to an SDE, which can be easily solved by many numerical methods. Comparing with the particle method, our's  avoids a great amount simulations of high dimensional sample trajectories for determining the empirical measure. Instead, we numerically solve a deterministic  Fokker-Planck equation, which results in  a high precision approximate density function. In order to observe dynamical behavior of  sample trajectory of a particle, one need to solve a high dimensional SDE if using  interactive particle system to approximate the mean field SDE. However, we only need to solve an SDE after getting an approximate density function.

%

Now, we  describe our method in more details.
We first truncate the whole $\mathbb{R}^d$ space on to a bounded domain and construct  a truncated Fokker-Planck equation equipped with a homogeneous Dirichlet boundary condition.
Then, we apply an explicit-implicit difference scheme  to approach the truncated equation to obtain a numerical density function. We replace the true density function in the mean field SDE with the approximate one, and hence  obtain an SDE, with which we can easily simulate the approximate trajectories of the mean field SDE. Meanwhile,  we also investigate corresponding  error estimates.

The organization of this paper is as follows. In Section \ref{s3.2}, we introduce a class of mean field SDEs and assumptions. In Section \ref{s3.3}, we construct a truncated Fokker-Planck equation and apply an effective numerical method to approximate the density function. Based on the solution to the truncated Fokker-Planck equation, we construct an SDE  to approximate the mean field SDE and investigate the corresponding error estimate. In Section \ref{s3.4},  we construct an auxiliary SDE in terms of the numerical density function to approximate the mean field SDE. Then, we apply  Euler-Maruyama method to solve this SDE and provide  error estimates. In  Section \ref{s3.5}, We present some numerical experiments  to illustrate  the effectiveness of our method.

\section{Preliminary}\label{s3.2}
In this section, we introduce a class of  mean field SDEs  and some assumptions.

Let $d>0$ be an integer and $(\mathbb{R}^d,| \cdot|)$  be a $d$-dimensional Euclidean space.
By  $|\cdot|_{F}$ we  denote the Frobenius norm of  a matrix. Let $C^2_0 (\mathbb{R}^d)$ be a space of  twice continuously differentiable  functions with compact support. For any integer $k\geq0$, denote by $H^k(\mathbb{R}^d)$ the standard Sobolev space, where $H^0(\mathbb{R}^d)=L^2(\mathbb{R}^d)$.

Let $T>0$ be a real number and  $(\Omega,\mathcal{F},\mathbb{P})$ be a probability space.  Denote by  $W(t)$ $(t\in[0,T])$  an $m$-dimensional standard Brownian motion.
By $\mathcal{P}(\mathbb{R}^d)$ we denote the set of all probability measures on $\mathbb{R}^d$ and define a subset
\begin{equation*}
 \mathcal{P}_2(\mathbb{R}^d)=\left\{ \mu\in\mathcal{P}(\mathbb{R}^d):\int_{\mathbb{R}^d}|x|^2\mu(dx) <\infty\right\}.
\end{equation*}
Let $L^2(\Omega)$ be a space of $\mathbb{R}^d$-valued, $\mathcal{F}$ measurable
random variables $X$ satisfying $E|X|^2 < \infty$. By $\mathcal{L}_{X}$ we denote the law of $X\in L^2(\Omega)$.

Consider  a class of mean field SDEs
\begin{equation}\label{e3.2.1}
  dX(t)=F(t,X(t),\mathcal{L}_{X(t)})dt+\sigma(t,X(t))dW(t),\quad X(0)=X_0\in L^2(\Omega).
\end{equation}
Define a filtration $\mathcal{F}_t=\sigma(X_0,W(s),0\leq s \leq t)$.
By abbreviation, let  $\mu_t=\mathcal{L}_{X(t)}$ $(0\leq t\leq T)$ denote a family of laws  of  solution $X(t)$ to \eqref{e3.2.1}.

We assume

\begin{itemize}
  \item [(A1)] the function  $K=(K_1,\cdots,K_d)^T:[0, T] \times \mathbb{R}^d\times \mathbb{R}^d\to\mathbb{R}^{d}$ together with its derivatives up to second order are bounded and continuous. Furthermore, there exists a constant $L>0$ such that $\int_{\mathbb{R}^d}|K(t,x,y)|^2dy\leq L$ for all $(t,x)$.
  \item [(A2)] $F =(F_1,\cdots,F_d)^T: [0, T] \times \mathbb{R}^d \times \mathcal{P}(\mathbb{R}^d) \to \mathbb{R}^d$ is defined as
  \begin{equation*}
    F(t,x,\mu_t)=f(t,x)+\int_{\mathbb{R}^d}K(t,x,y)\mu_t(dy),
  \end{equation*}
  where $f=(f_1,\cdots,f_d)^T:[0, T] \times \mathbb{R}^d \to \mathbb{R}^d$ is twice  continuously differentiable with bounded first and seconde derivatives.
  \item [(A3)]$\sigma: [0, T] \times \mathbb{R}^d\to\mathbb{R}^{d\times m}$ is  twice continuously differentiable with  bounded first and second derivatives.
\end{itemize}

From  (A2) and (A3), it follows that there exist a constant $L>0$ such that
\begin{equation}\label{e3.2.5}
    \begin{aligned}
    |f(t,x)-f(t,\bar{x})|+ |\sigma(t,{x})-\sigma(t,\bar{x})|_F&\leq L|x-\bar{x}|,\quad \forall t\in[0,T],\ x, \bar{x}\in\mathbb{R}^d.
    \end{aligned}
  \end{equation}
Under above three assumptions and by \cite{Eva2013,KumNeeReiSto2022}, \eqref{e3.2.1} has a unique solution $X(t)$  satisfying
\begin{equation}\label{e3.2.7}
  E|X(t)|^2\leq L(1+E|X_0|^2),\quad \forall t\in[0,T].
\end{equation}

Generally speaking, the existing numerical methods for approximating the mean field SDE are based on converting \eqref{e3.2.1} to an interacting particle system, where the true measure is replaced by an empirical measure, which is a linear combination of point measures and is independently simulated by Monte Carlo method. Then, the interacting particle system is governed by a high dimensional SDE  and many numerical methods can be applied. However, in order to get a high-accuracy empirical measure, the numbers of particles and sample trajectories should be extremely enlarged, which involves high computation complexity.

In this paper, we will propose a completely different method from the particle method to approach \eqref{e3.2.1}. In fact, the density function of  the solution to \eqref{e3.2.1} satisfies a deterministic nonlinear Fokker-Planck equation. Our strategy is to apply finite difference method to solve the nonlinear Fokker-Planck equation so as to obtain an approximate density function, which helps us to transform \eqref{e3.2.1} to an SDE. By virtue of  the classical theory on numerical partial differential equation (PDE for short), we obtain an  approximate density function with desired precision, and then  set up a high-accuracy approximation to  \eqref{e3.2.1}. Compared to Monte Carlo simulation, our method reduces computation complexity while approximating the law of solution and accelerates the procedure of solving \eqref{e3.2.1}.

\section{A nonlinear Fokker-Planck equation and the approximation}\label{s3.3}
In this section, we first derive a nonlinear Fokker-Planck equation which characterizes the development of   the density function of solution to   mean field SDE \eqref{e3.2.1}. Then, we truncate the nonlinear Fokker-Planck onto a bounded domain and assign a homogeneous Dirichlet boundary value condition  to construct an approximate equation.  Afterward, we apply an explicit-implicit finite  difference method  to  set up a fully discretized scheme for the truncated equation. Based on the approximate density function we construct an  SDE to approximate the mean field SDE \eqref{e3.2.4}. We will derive error estimates between the solutions to the approximate SDE and the original mean field SDE.

\subsection{A nonlinear Fokker-Planck equation}
Denote by  $p(t,x)$ the density function of the solution $X(t)$ to mean field SDE \eqref{e3.2.4}. Without loss of generality, we assume $p(t,\cdot)\in H^2(\mathbb{R}^d)$ and $p\geq 0$. Thus, we have $\mu_t(dx) =p(t, x)dx$. Let $p_0(x)$ be the  density function of $X_0$.  From \eqref{e3.2.7}, it follows that
\begin{equation}\label{e3.3.1}
  \int_{\mathbb{R}^d}|x|^2p(t,x)dx\leq C(1+\int_{\mathbb{R}^d}|x|^2p_0(x)dx)<+\infty,
\end{equation}
which implies $p(t,x)=o(\frac{1}{|x|^3})$ as $x\to \infty$.

Now, we are already to derive a nonlinear Fokker-Planck equation.
For any $\phi \in C^2_0 (\mathbb{R}^d)$, let $Y(t) = \phi(X(t))$, then
\begin{equation}\label{e3.2.10}
  E(Y(t) ) = \int_{\mathbb{R}^d}p(t,x)\phi(x) dx.
\end{equation}
Thus, \eqref{e3.2.1} equivalently becomes
\begin{equation}\label{e3.2.4}
  dX(t)=f(t,X(t))dt+\int_{\mathbb{R}^d}K(t,X(t),y)p(t,y)dydt+\sigma(t,X(t))dW(t).
\end{equation}
By It\'o's formula, we have
\begin{equation}\label{e3.2.6}
  \phi(X(t)) = \phi(X(0))+\int_{0}^tL_{\mu_s}\phi (s,X(s))ds+\int_0^t (\bigtriangledown \phi(X(s)))^T\sigma(s,X(s)) dW(s),
\end{equation}
where
\begin{equation*}
  L_{\mu_t}\phi(t,x)=\sum_{i=1}^d\Big(f_i(t,x)+\int_{\mathbb{R}^d}K_i(t,x,y)p(t,y)dy\Big)\frac{\partial \phi(x)}{\partial x_i}+\frac12\sum_{i,j=1}^d a_{ij}(t,x)\frac{\partial^2 \phi(x)}{\partial x_i\partial x_j}.
\end{equation*}
Here $A=(a_{ij}):= \sigma\sigma^T$ is a symmetric matrix-valued function.
Take expectations on both sides of \eqref{e3.2.6} and apply the formula of integration by parts, then we get
\begin{equation*}\label{e3.2.8}
\begin{aligned}
&\int_{\mathbb{R}^d}p(t,x)\phi(x)dx=\int_{\mathbb{R}^d}p_0(x)\phi(x)dx\\
&\quad -\int_{\mathbb{R}^d}\!\sum_{i=1}^d\!\int_{0}^t\frac{\partial [(f_i(s,x)\!+\!\int_{\mathbb{R}^d}K_i(s,x,y)p(s,y)dy)p(s,x)]}{\partial x_i}ds\phi(x)dx\\
&\quad +\frac12\int_{\mathbb{R}^d}\sum_{i,j=1}^d\int_{0}^t\frac{\partial^2 [a_{ij}(s,x)p(s,x)]}{\partial x_i\partial x_j}ds\phi(x)dx.
\end{aligned}
\end{equation*}
Since $C_0^2(\mathbb{R}^d)$ is dense in $L^2(\mathbb{R}^d)$, we obtain
\begin{equation}\label{e3.2.3}
\begin{aligned}
p(t,x)
=&p_0(x)-\sum_{i=1}^d\int_{0}^t
\frac{\partial [(f_i(s,x)+\int_{\mathbb{R}^d}K_i(s,x,y)p(s,y)dy)p(s,x)]}{\partial x_i}ds\\
&+\frac12\sum_{i,j=1}^d\int_{0}^t\frac{\partial^2 [a_{ij}(s,x)p(s,x)]}{\partial x_i\partial x_j}ds.
\end{aligned}
\end{equation}
Notice that the functions $f_i$, $K_i$ and $a_{ij}$ are continuous with respect to $t$,  the solution $p$ of above equation  is  continuously differentiable with respect to $t$. Thus,  we obtain a nonlinear  Fokker-Planck equation \cite{Fra2005,BarRoc2020}
\begin{equation}\label{e3.2.2}
\frac{\partial p}{\partial t}=-\sum_{i=1}^d\frac{\partial[(f_i+\int_{\mathbb{R}^d}K_i(\cdot,\cdot,y)p(\cdot,y)dy)p]}{\partial x_i}+\frac12\sum_{i,j=1}^d\frac{\partial^2[a_{ij}p]}{\partial x_i\partial x_j},
\end{equation}
with initial value condition $p(0, x) = p_0(x)$.

In order to ensure the well-posedness of \eqref{e3.2.2}, we also assume
\begin{itemize}
\item [(A4)] There exist two constants $0 < \gamma_1 \leq \gamma_2$ such that
\begin{equation*}
  \gamma_1|y|^2\leq y^T A(t, x)y \leq  \gamma_2|y|^2,\quad \forall t \in[0,T], x, y \in\mathbb{ R}^d.
\end{equation*}
\end{itemize}


Under (A1)-(A4), the well-posedness and regularity of solution to autonomous nonlinear Fokker-Planck equation have been established \cite{BarRoc2020}, which may not be directly applied to the non-autonomous case.  However, we focus on  constructing a new numerical method to approximate the mean field SED \eqref{e3.2.4} with the help of a solution to \eqref{e3.2.2}. In fact, the  density function $p(t,x)$ is  a solution to \eqref{e3.2.2} and hence  we only assume that $p$ is the unique solution  without proof.

\subsection{A numerical approximation to nonlinear Fokker-Planck equation}\label{s3.3.2}
In order to approximate the nonlinear Fokker-Planck equation \eqref{e3.2.2}, we first truncate $\mathbb{R}^d$ to  a bounded domain $\mathcal D=(-\alpha,\alpha)^d$ for some $\alpha>0$.
Then we construct a truncated equation on $[0,T]\times\mathcal D$ to approximate \eqref{e3.2.2}
\begin{equation}\label{e3.3.3}
\begin{aligned}
  &\frac{\partial p_{\mathcal D}}{\partial t}= -\sum_{i=1}^d\frac{\partial [(f_i+\int_{\mathcal D}K_i(\cdot,\cdot,y)p_{\mathcal D}(\cdot,y)dy)p_{\mathcal D}]}{\partial x_i}+\frac12\sum_{i,j=1}^d\frac{\partial^2[a_{ij}p_{\mathcal D}]}{\partial x_i\partial x_j}=\\
  &  -\sum_{i=1}^d\frac{\partial [(f_i\!+\!\int_{\mathcal D}K_i(\cdot,\cdot,y)p_{\mathcal D}(\cdot,y)dy)p_{\mathcal D}]}{\partial x_i}\!+\!\frac12\sum_{i=1}^d\frac{\partial^2[a_{ii}p_{\mathcal D}]}{\partial x_i ^2}\!+\!\frac12\sum_{i\not=j=1}^d\frac{\partial^2[a_{ij}p_{\mathcal D}]}{\partial x_i\partial x_j},
  \end{aligned}
\end{equation}
with initial-boundary value condition
\begin{equation}\label{e3.3.4}
  p_{\mathcal D}(0,x)=p_{0}|_{{\mathcal D}}(x) \text{ for } x\in {\mathcal D},\quad p_{\mathcal D}(t,x)=0 \text{ for } t\in(0,T),\ x\in \partial {\mathcal D}.
\end{equation}

The unique existence and regularity estimates of the solution to \eqref{e3.3.3}-\eqref{e3.3.4} are interesting problems, but it is not the main task of this paper. Hence, we assume that the initial-boundary value problem \eqref{e3.3.3}-\eqref{e3.3.4} has  a unique positive  solution $p_{\mathcal D}(t,\cdot)\in H^2({\mathcal D})\cap H_0^1(\mathcal D)$. Then, $p_{\mathcal D}$  is an approximation to  $p|_{\mathcal D}$ and $p_{\mathcal D}\to p$ as $\alpha\to \infty$ in some a suitable space, which will be described in detail later.

Now, we study a finite difference approximation  to \eqref{e3.3.3}.
For any integer $N>0$, let  $\kappa=\frac{T }{N}$ be a time step size  and $t_n=n\kappa$ $(n=0,1,\cdots,N)$   be partition nodes.
For any integer $M>0$, define two index sets
\begin{equation*}
	\begin{aligned}
		&\mathbb{Z}_M=\{k\in Z^d: |k_i|\leq M,i=1,\cdots,d\},\quad
		\Theta_M=\mathbb{Z}_M\setminus \mathbb{Z}_{M-1}.
	\end{aligned}
\end{equation*}
Let  $h=\frac{\alpha}{M}$ be a spatial meshsize and   $x^k=hk$, $k\in \mathbb{Z}_M$, be partition nodes  inside domain $\mathcal D$. Define  ${D}_k=[x_{k_1},x_{k_1+1}]\times\cdots\times[x_{k_d},x_{k_d+1}]$, then $\{D_k: k\in\mathbb{Z}_M, k_i\leq M-1, i=1,\cdots, d\}$ is a uniform partition of ${\mathcal D}$.

Define a space of finite sequences
\begin{equation*}
  S_{\mathbb{Z}_{M}}^0=\left\{v_{\mathbb{Z}_{M}}=(v^k)_{k\in\mathbb{Z}_{M}}\in\mathbb{R}^{\mathbb{Z}_{M}}:
v^k=0~\text{for}~k\in\Theta_M\right\}
\end{equation*}
equipped with a discrete $L^2$-norm
$
\|v_{\mathbb{Z}_{M}}\|^2_{S^0_{\mathbb{Z}_{M}}}=\displaystyle\sum_{k\in \mathbb{Z}_{M}}|v^k|^2h^d.
$

For any $1\leq i\leq d$, define a unit vector $\eta^i\in\mathbb{R}^d$ whose i-th component is equal to $1 $. For any  $n=0,1,\cdots, N$ and $k\in \mathbb{Z}_{M}$, let $f_{i}^{n,k}=f_i(t_n,x^k)$,  $a_{ij}^{n,k}=a_{ij}(t_n,x^k)$ and $p^{n,k}_{{\mathcal D}}$ be an approximation to $p_{\mathcal D}(t_n,x^k)$.  Denote a numerical integration by
${\mathcal S}_{i}^{n,k}=\displaystyle\sum_{s\in\mathbb{Z}_M} K_i(t_n,x^k,x^{s})p^{n,s}_{{\mathcal D}}h^d$. Now, we construct an explicit-implicit difference scheme  to approximate  \eqref{e3.3.3}
\begin{equation}\label{e3.2.15}
	\begin{aligned}\displaystyle
		&\frac{p^{n+1,k}_{{\mathcal D}}-p^{n,k}_{{\mathcal D}}}{\kappa}\\ =&-\sum_{i=1}^{d}\frac{(f_{i}^{n,k+\eta^{i}}+{\mathcal S}_{i}^{n,k+\eta^{i}})p^{n,k+\eta^{i}}_{{\mathcal D}}
-(f_{i}^{n,k-\eta^{i}}+{\mathcal S}_{i}^{n,k-\eta^{i}})p^{n,k-\eta^{i}}_{{\mathcal D}}}{2h}\\
	&+\sum_{i=1}^{d}\frac{a_{ii}^{n+1,k+\eta^{i}}p^{n+1,k+\eta^{i}}_{{\mathcal D}}-2a_{ii}^{n+1,k}p^{n+1,k}_{{\mathcal D}}
			+a_{ii}^{n+1,k-\eta^{i}}p^{n+1,k-\eta^{i}}_{{\mathcal D}}}{2h^2}\\
		&+\sum_{i\neq j=1}^{d}\Big(\frac{a_{ij}^{n+1,k+\eta^{i}+\eta^{j}}p^{n+1,k+\eta^{i}+\eta^{j}}_{{\mathcal D}}-a_{ij}^{n+1,k+\eta^{i}-\eta^{j}}p^{n+1,k+\eta^{i}-\eta^{j}}_{{\mathcal D}}}{8h^2}\\
&\quad\quad+\frac{-a_{ij}^{n+1,k-\eta^{i}+\eta^{j}}p^{n+1,k-\eta^{i}+\eta^{j}}_{{\mathcal D}}
			+a_{ij}^{n+1,k-\eta^{i}-\eta^{j}}p^{n+1,k-\eta^{i}-\eta^{j}}_{{\mathcal D}}}{8h^2}\Big).
	\end{aligned}
\end{equation}
As  $\kappa$ and $h$ are sufficiently small, we  obtain a unique $p_{{\mathcal D},\mathbb{Z}_{M}}^{n+1}=(p^{n+1,k}_{{\mathcal D}})_{k\in\mathbb{Z}_{M}}\in S^0_{\mathbb{Z}_{M}}$ by solving linear equation \eqref{e3.2.15} for any given $p_{{\mathcal D},\mathbb{Z}_{M}}^n=(p^{n,k}_{{\mathcal D}})_{k\in\mathbb{Z}_{M}}\in S^0_{\mathbb{Z}_{M}}$.
Hence for a given initial value $(p_{{\mathcal D}}(0,x^k))_{k\in\mathbb{Z}_{M}}\in S^0_{\mathbb{Z}_{M}}$, we  recursively solve \eqref{e3.2.15} to get  a  numerical solution $p_{{\mathcal D},\mathbb{Z}_{M}}^n$  $(n=1,2,\cdots, N)$ to \eqref{e3.3.3}-\eqref{e3.3.4}. The classical error analysis theorem  \cite{LeV2007} provides convergence property, which is summarized in next lemma and its proof is omitted.

\begin{lemma}\label{le3.3.1}
Assume (A1)-(A4) hold.    Let $ p_{\mathcal D}$  and  $ p_{{\mathcal D},\mathbb{Z}_{M}}^n$ $(n=1,2,\cdots, N)$ be the real  and numerical solutions to  \eqref{e3.3.3}-\eqref{e3.3.4},  respectively. Then, there exits a constant $C > 0$ such that
\begin{equation*}
	\|(p_{\mathcal D}(t_n,x^k))_{k\in\mathbb{Z}_{M}}- p_{{\mathcal D},\mathbb{Z}_{M}}^n\|_{S^0_{\mathbb{Z}_{M}}}\leq C(\kappa+h^2).
\end{equation*}
\end{lemma}

\subsection{An approximation to the mean field equation}
In this section, we apply the approximate density function $p_{\mathcal D}(t,x)$ to construct an SDE to approximate the mean field SDE \eqref{e3.2.4}. We also derive error estimates between their solutions.

Let  $I_{\mathcal D}(x)$ be an indicator function on $\mathcal D$, i.e. $I_{\mathcal D}(x) = 1$ if $x \in{\mathcal D}$ and $I_{\mathcal D}(x) = 0$ otherwise.  By $p_{\mathcal D} \cdot I_{\mathcal D}(x)$ we denote a function defined on the whole space $\mathbb{R}^d$, which can be regarded as the zero extension of $p_{\mathcal D}$. If no confusion occurs, we still use $p_{\mathcal D}$ to denote the extended function $p_{\mathcal D} \cdot I_{\mathcal D}(x)$.  Furthermore, we also assume $p_{\mathcal D}$ together with its derivatives converges to the counterparts of the solution $p$ to  \eqref{e3.2.2}  as $\alpha\to\infty$, i.e.\
\begin{equation*}
  \Phi(\alpha):=\displaystyle\sup_{t\in[0,T]}\|p(t,\cdot)-p_{\mathcal D}(t,\cdot)\|_{L^1(\mathbb{R}^d)}+\|p(t,\cdot)-p_{\mathcal D}(t,\cdot)\|_{H^2(\mathbb{R}^d)}\to 0 \text{ as } \alpha\to \infty.
  \end{equation*}
Then  there exists a constant $C>0$ such that for $t\in[0,T]$ and large $\alpha$
\begin{equation}\label{e3.3.6}
\begin{aligned}
 \int_{\mathcal D}p_{\mathcal D}(t,x)dx&\leq\int_{\mathbb{R}^d}p(t,x)dx+\int_{\mathbb{R}^d}|p(t,x)-p_{\mathcal D}(t,x)|dx\leq1+\Phi(\alpha)\leq C,\\
 \|p_{\mathcal D}(t,\cdot)\|_{H^2(\mathcal D)}&\leq \|p(t,\cdot)\|_{H^2(\mathbb{R}^d)}+\|p(t,\cdot)-p_{\mathcal D}(t,\cdot)\|_{H^2(\mathbb{R}^d)}\leq C+\Phi(\alpha)\leq C.
 \end{aligned}
\end{equation}

We use the approximate density function $p_{\mathcal D}$ to replace the real density function $p$ in mean field SDE \eqref{e3.2.4}, and hence obtain an SDE
\begin{equation}\label{e3.3.8}
dX_{\mathcal D}(t)=f(t, X_{\mathcal D}(t))dt+\int_{\mathbb{R}^d}K(t,X_{\mathcal D}(t),y) p_{\mathcal D}(t,y)dydt+\sigma(t, X_{\mathcal D}(t))dW(t),
\end{equation}
with the same initial value condition $X_{\mathcal D}(0)=X(0)$. Obviously,  \eqref{e3.3.8} is an approximation to the mean field SDE \eqref{e3.2.4}.
Under assumptions (A1)-(A3), \eqref{e3.3.8} has a unique solution $X_{\mathcal D}(t)$  satisfying a priori estimate, cf.\ \cite{Eva2013}
\begin{equation}\label{e3.3.11}
 E|X_{\mathcal D}(t)|^2\leq C(1+E|X_0|^2),\quad \forall t\in[0,T].
\end{equation}

The next lemma investigates the error estimate between the solutions of mean field SDE \eqref{e3.2.4} and SDE \eqref{e3.3.8}.

\begin{lemma}\label{le3.3.2}
Assume (A1)-(A3) hold. Let $X(t)$ and $X_{\mathcal D}(t)$ be solutions of \eqref{e3.2.4} and \eqref{e3.3.8}, respectively. Then, there holds
\begin{equation*}
  E|X(t)-X_{\mathcal D}(t)|^2\leq C\Phi^2(\alpha),\quad \forall t\in[0,T],
\end{equation*}
where $C=C(L,T)>0$ is a constant.
\end{lemma}

\begin{proof}
According to \eqref{e3.2.4} and \eqref{e3.3.8}, we have
\begin{equation*}
\begin{aligned}
  &{X}(t)-{X}_{\mathcal D}(t)=\int_0^tf(s,{X}(s))-f(s,{X}_{\mathcal D}(s))ds\\
  &+\int_0^t\int_{\mathbb{R}^d}\!K(s,{X}(s),y){p}(s,y)-K(s,{X}_{\mathcal D}(s),y){p}_{\mathcal D}(s,y)dyds\\
  &+\int_0^t\sigma(s,{X}(s))- \sigma(s,{X}_{\mathcal D}(s))dW(s).\\
    \end{aligned}
\end{equation*}
By It\'o's formula, we get
\begin{equation*}
\begin{aligned}
  &|{X}(t)-{X}_{\mathcal D}(t)|^2=2\int_0^t\langle {X}(s)-{X}_{\mathcal D}(s),f(s, {X}(s))-f(s, {X}_{\mathcal D}(s))\rangle ds\\
  &+2\int_0^t\langle {X}(s)-{X}_{\mathcal D}(s),\int_{\mathbb{R}^d}K(s, {X}(s),y){p}(s,y)-K(s, {X}_{\mathcal D}(s),y){p}_{\mathcal D}(s,y)dy\rangle ds\\
  &+\int_0^t \text{tr}[(\sigma(s, {X}(s))-\sigma(s, {X}_{\mathcal D}(s)))(\sigma(s, {X}(s))-\sigma(s, {X}_{\mathcal D}(s)))^T]ds\\
  &+2\int_0^t \langle{X}(s)-{X}_{\mathcal D}(s),(\sigma(s,{X}(s))- \sigma(s,{X}_{\mathcal D}(s))dW(s)\rangle.
\end{aligned}
\end{equation*}
Taking  expectations on both sides of above equation and applying \eqref{e3.2.5} and Cauchy-Schwarz inequality, we obtain
\begin{equation*}
\begin{aligned}
  &E|{X}(t)-{X}_{\mathcal D}(t)|^2=2E\int_0^t\langle {X}(s)-{X}_{\mathcal D}(s),f(s, {X}(s))-f(s, {X}_{\mathcal D}(s))\rangle ds\\
  &+2E\int_0^t\langle {X}(s)-{X}_{\mathcal D}(s),\int_{\mathbb{R}^d}K(s, {X}(s),y){p}(s,y)-K(s, {X}_{\mathcal D}(s),y){p}_{\mathcal D}(s,y)dy\rangle ds\\
  &+E\int_0^t\text{tr}[(\sigma(s, {X}(s))-\sigma(s, {X}_{\mathcal D}(s)))(\sigma(s, {X}(s))-\sigma(s, {X}_{\mathcal D}(s)))^T]ds\\
  \leq&2E\int_0^t|{X}(s)-{X}_{\mathcal D}(s)|\cdot |f(s, {X}(s))-f(s, {X}_{\mathcal D}(s))| ds\\
  &+2E\int_0^t| {X}(s)-{X}_{\mathcal D}(s)|\cdot |\int_{\mathbb{R}^d}K(s, {X}(s),y){p}(s,y)-K(s, {X}_{\mathcal D}(s),y){p}_{\mathcal D}(s,y)dy| ds\\
  &+E\int_0^t|\sigma(s, {X}(s))-\sigma(s, {X}_{\mathcal D}(s))|_F^2ds\\
  \leq& 2E\int_0^tL|{X}(s)-{X}_{\mathcal D}(s)|^2ds+E\int_0^t|{X}(s)-{X}_{\mathcal D}(s)|^2ds\\
  &+E\int_0^t|\int_{\mathbb{R}^d}K(s, {X}(s),y){p}(s,y)-K(s, {X}_{\mathcal D}(s),y){p}_{\mathcal D}(s,y)dy|^2 ds\\
  &+E\int_0^t L^2|{X}(s)-{X}_{\mathcal D}(s)|^2ds.
\end{aligned}
\end{equation*}
Noticing that
\begin{equation*}
\begin{aligned}
  &|\int_{\mathbb{R}^d}K(s, {X}(s),y){p}(s,y)-K(s, {X}_{\mathcal D}(s),y){p}_{\mathcal D}(s,y)dy|\\
  \leq&\int_{\mathbb{R}^d}|K(s, {X}(s),y)-K(s, {X}_{\mathcal D}(s),y)|p(s,y)dy  \\
  &+\int_{\mathbb{R}^d}|K(s, {X}_{\mathcal D}(s),y)|\cdot | {p}(s,y)-{p}_{\mathcal D}(s,y)|dy\\
  \leq&L|{X}(s)-{X}_{\mathcal D}(s)|\int_{\mathbb{R}^d}p(s,y)dy+C\Phi(\alpha)
  \leq L|{X}(s)-{X}_{\mathcal D}(s)|+C\Phi(\alpha).
\end{aligned}
\end{equation*}
The above two inequalities imply
\begin{equation*}
\begin{aligned}
  E|X(t)-{X}_{\mathcal D}(t)|^2\leq C(L)\int_0^tE|{X}(s)-{X}_{\mathcal D}(s)|^2 ds+C(L)T\Phi^2(\alpha).
  \end{aligned}
\end{equation*}
 Then the proof follows from Gronwall's inequality.
\end{proof}

\section{A numerical approximation to the mean field SDE}\label{s3.4}
In this section, based on the numerical solution to \eqref{e3.3.3}-\eqref{e3.3.4}, we construct an auxiliary SDE to approximate \eqref{e3.3.8} and investigate the corresponding error estimates. Then, we apply Euler-Maruyama method to the auxiliary SDE to set up a numerical scheme which is used to approximate the  mean field SDE \eqref{e3.2.4}.  Finally, we  study the error estimates between the numerical solution  and the exact solution of the mean field SDE.

\subsection{An auxiliary SDE and error estimates}\label{s3.4.1}
%
Define a piece-wise constant function  $p_{\kappa,h}(t,x)={p}_{{\mathcal D}}^{n,k}$ for $(t,x)\in [t_n,t_{n+1})\times D_k$ and extend its domain  to $\mathbb{R}^d$ by $p_{\kappa,h}(t,x)=0$ for $x\in {\mathcal D}^c$. We replace the approximate density function $p_{\mathcal D}(t,x)$ in \eqref{e3.3.8} by $p_{\kappa,h}(t,x)$, and then set up an auxiliary SDE for $t\in[0,T]$
\begin{equation}\label{e3.3.2}
  d \hat{X}(t)=f(t, \hat{X}(t))dt+\int_{\mathbb{R}^d}K(t,\hat{X}(t),y)p_{\kappa,h}(t,x)dydt+\sigma(t, \hat{X}(t))dW(t),
\end{equation}
with  $\hat{X}(0)=X(0)$.
Under assumptions (A1)-(A3), \eqref{e3.3.2} has a unique solution $\hat X(t)$ which has  second order moment estimate, cf.\ \cite{Eva2013}
\begin{equation}\label{e3.3.10}
  E|\hat X(t)|^2\leq C(1+E|X_0|^2),\quad \forall t\in[0,T],
\end{equation}
where $C>0$ is a constant.

\begin{lemma}\label{le3.3.3}
Assume (A1)-(A4) hold. Let $X_{\mathcal D}(t)$ and $\hat X(t)$ be solutions of \eqref{e3.3.8} and \eqref{e3.3.2}, respectively. Then, there holds
\begin{equation*}
  E|X_{\mathcal D}(t)-\hat X(t)|^2\leq C(\kappa^2+h^2),\quad \forall t\in[0,T],
\end{equation*}
where $C=C(L,T)>0$ is a constant.
\end{lemma}

\begin{proof}
\eqref{e3.3.8} minus \eqref{e3.3.2} yields
\begin{equation*}
\begin{aligned}
  X_{\mathcal D}(t)-\hat{X}(t)=&\int_0^tf(s,X_{\mathcal D}(s))-f(s,\hat{X}(s))ds\\
  &+\int_0^t\int_{\mathbb{R}^d}K(s,X_{\mathcal D}(s),y){p}_{\mathcal D}(s,y)-K(s,\hat{X}(s),y)p_{\kappa,h}(s,y)dyds\\
  &+\int_0^t\sigma(s,X_{\mathcal D}(s))- \sigma(s,\hat{X}(s))dW(s),
    \end{aligned}
\end{equation*}
By It\'o's formula, we get
\begin{equation*}
\begin{aligned}
  &|X_{\mathcal D}(t)-\hat{X}(t)|^2=2\int_0^t\langle X_{\mathcal D}(s)-\hat{X}(s),f(s, {X}_{\mathcal D}(s))-f(s, \hat{X}(s))\rangle ds\\
  &+2\int_0^t\langle X_{\mathcal D}(s)-\hat{X}(s),\int_{\mathbb{R}^d}K(s, {X}_{\mathcal D}(s),y){p}_{\mathcal D}(s,y)-K(s, \hat{X}(s),y)p_{\kappa,h}(s,y)dy\rangle ds\\
  &+\int_0^t\text{tr}[(\sigma(s, X_{\mathcal D}(s))-\sigma(s, \hat{X}(s)))(\sigma(s, X_{\mathcal D}(s))-\sigma(s, \hat{X}(s)))^T]ds\\
  &+2\int_0^t \langle X_{\mathcal D}(s)-\hat{X}(s),(\sigma(s, X_{\mathcal D}(s))-\sigma(s, \hat{X}(s)))dW(s)\rangle.
\end{aligned}
\end{equation*}
Similar to the proof in Lemma \ref{le3.3.2} and by cauchy-schwarz inequality, we get
\begin{equation*}
\begin{aligned}
 &E|X_{\mathcal D}(t)-\hat{X}(t)|^2\\
  \leq& CE\int_0^t |X_{\mathcal D}(s)-\hat{X}(s)|^2ds\\
  &+E\int_0^t|\int_{\mathbb{R}^d}K(s, {X}_{\mathcal D}(s),y){p}_{\mathcal D}(s,y)-K(s, \hat{X}(s),y)p_{\kappa,h}(s,y)dy|^2 ds\\
  \leq&CE\int_0^t |X_{\mathcal D}(s)-\hat{X}(s)|^2ds\\
  &+E\int_0^t2\big(\int_{\mathbb{R}^d}|K(s, {X}_{\mathcal D}(s),y)-K(s, \hat{X}(s),y)|{p}_{\mathcal D}(s,y)dy \big)^2ds\\
  &+E\int_0^t2\big(\int_{\mathbb{R}^d}|K(s, \hat{X}(s),y)|\cdot | {p}_{\mathcal D}(s,y)-p_{\kappa,h}(s,y)|dy\big)^2ds\\
  \leq&CE\int_0^t |X_{\mathcal D}(s)-\hat{X}(s)|^2ds\\
  &+E\int_0^t2L^2|{X}(s)-{X}_{\mathcal D}(s)|^2\big(\int_{\mathcal D}{p}_{\mathcal D}(s,y)dy\big)^2ds\\
  &+E\int_0^t2\int_{\mathbb{R}^d}|K(s, \hat{X}(s),y)|^2dy\int_{\mathcal D}| {p}_{\mathcal D}(s,y)-p_{\kappa,h}(s,y)|^2dyds,
\end{aligned}
\end{equation*}
where $C=C(L)>0$ is a constant.

By mean value theorem, Lemma \ref{le3.3.1} and \eqref{e3.3.6}, we have for any $n=0,1\cdots N-1$ and  $s\in[t_n,t_{n+1}]$
\begin{equation*}
\begin{aligned}
  &\int_{\mathcal D}|{p}_{\mathcal D}(s,y)-p_{\kappa,h}(s,y)|^2dy\\
  =&\sum_{k\in\mathbb{Z}_M}\int_{D_k}|{p}_{\mathcal D}(s,y)-{p}_{\mathcal D}^{n,k}|^2dy\\
  \leq&2\sum_{k\in\mathbb{Z}_M}\int_{D_k}|{p}_{\mathcal D}(s,y)-{p}_{\mathcal D}(t_n,x^k)|^2dy+2\sum_{k\in\mathbb{Z}_M}\int_{D_k}|{p}_{\mathcal D}(t_n,x^k)-{p}_{\mathcal D}^{n,k}|^2dy\\
  \leq&2\sum_{k\in\mathbb{Z}_M}\int_{D_k}|\frac{\partial p_{\mathcal D}(t_{n}^\theta,y^k_\theta )}{\partial t}(s-t_n)+\sum_{i=1}^d\frac{\partial p_{\mathcal D}(t_{n}^\theta,y^k_\theta )}{\partial y_i}(y_i-x^{k}_i)|^2dy\\
  &+2\|(p_{\mathcal D}(t_n,x^k))_{k\in\mathbb{Z}_{M}}- p_{{\mathcal D},\mathbb{Z}_{M}}^n\|_{S^0_{\mathbb{Z}_{M}}}^2\\
   \leq&2(d+1) \int_{\mathcal D}\kappa^2|\frac{\partial p_{\mathcal D}(t_{n}^\theta,y^k_\theta )}{\partial t}|^2dy+h^2\sum_{i=1}^d|\frac{\partial p_{\mathcal D}(t_{n}^\theta,y^k_\theta )}{\partial y_i}|^2dy+C(\kappa^2+h^4)\\
  \leq& C(\kappa^2+h^2),
\end{aligned}
\end{equation*}
where $(t_{n}^\theta, y^k_\theta) = \theta(s, y) + (1-\theta)(t_n, x^k)$ and $C>0$ is a constant independent of $\kappa$ and $h$.
The above two inequalities imply
\begin{equation*}
\begin{aligned}
E|X_{\mathcal D}(t)-\hat{X}(t)|^2
\leq C(\kappa^2+h^2)+ C\int_0^tE|X_{\mathcal D}(s)-\hat{X}(s)|^2 ds.
\end{aligned}
\end{equation*}
 Then the proof follows from Gronwall's inequality.
\end{proof}

\subsection{A numerical approximation to the auxiliary SDE}
In this section,  we apply Euler-Maruyama method to auxiliary equation \eqref{e3.3.2}   to construct a numerical scheme for the original mean field SDE \eqref{e3.2.1} and derive error estimates.

We take the same temporal step size $\kappa=\frac{T }{N}$ and  $t_{n}=n\kappa$ $(n=0,1,\cdots,N)$  as in section \ref{s3.3.2} and apply Euler-Maruyama scheme for  SDE \eqref{e3.3.2}
\begin{equation}\label{e3.4.4}
\begin{aligned}
  \hat{X}^{n+1} =& \hat{X}^{n}\!+\! f(t_n,\hat{X}_{\kappa}^{n})\kappa\!+\!\kappa\!\int_{\mathbb{R}^d}\!K(t_n,\hat{X}^{n},y)p_{\kappa,h}(t_n,y) dy\!+\!\sigma(t_n,\hat{X}^{n})\Delta W_n.
  \end{aligned}
\end{equation}
where  $\Delta W_n = W(t_{n+1})-W(t_{n})$. From \cite{KloPla1992,Sau2013,LordPoweShard2014}, it follows that
\begin{equation}\label{e3.4.2}
  E|\hat{X}(t_n)-\hat{X}^{n}|^2\leq C\kappa,\quad  n=1,2,\cdots, N,
\end{equation}
where $\hat{X}(t)$ is the solution of auxiliary equation \eqref{e3.3.2}. Specifically, if the function $\sigma$ is independent of $(t,x)$, where  \eqref{e3.4.4} is driven by additive noise, the convergence order for Euler-Maruyama method becomes of $1$, i.e.\
\begin{equation}\label{e3.4.3}
  E|\hat{X}(t_n)-\hat{X}^{n}|^2\leq C\kappa^2,\quad  n=1,2,\cdots, N.
\end{equation}

Now, we are ready to state and  prove the main result in this paper.

\begin{theorem}
Assume (A1)-(A4) hold.  Let $ X(t)$ and $\hat{X}^n$ $(n=0,1,\cdots, N)$ be solutions to \eqref{e3.2.4} and \eqref{e3.4.4}, respectively. Then there exists a constant $C>0$ independent of $\kappa$ and $h$ such that
\begin{equation*}
  E|{X}(t_n)-\hat{X}^n|^2\leq C\Phi^2(\alpha)+C(\kappa+h^2),\quad n=1,2,\cdots, N.
\end{equation*}
Furthermore, if the function $\sigma$ is independent of $(t,x)$, there holds
\begin{equation*}
  E|{X}(t_n)-\hat{X}^n|^2\leq C\Phi^2(\alpha)+C(\kappa^2+h^2),\quad n=1,2,\cdots, N.
\end{equation*}
\end{theorem}

\begin{proof}
Let ${X}_{\mathcal D}(t)$ and $\hat{X}(t)$ be solutions to \eqref{e3.3.8} and \eqref{e3.3.2}, respectively. According to Lemmas \ref{le3.3.2} and \ref{e3.3.2} and estimate in \eqref{e3.4.2}, we have
\begin{equation*}
\begin{aligned}
  &E|{X}(t_n)-\hat{X}^n|^2\\
  \leq& 3E|X(t_n)-{X}_{\mathcal D}(t_n)|^2+
  3E|{X}_{\mathcal D}(t_n)-\hat{X}(t_n)|^2+3E|\hat{X}(t_n)-\hat X^n|^2\\
  \leq& C\Phi^2(\alpha)+C(\kappa^2+h^2)+C\kappa\\
  \leq& C\Phi^2(\alpha)+C(\kappa+h^2).
  \end{aligned}
\end{equation*}

The other estimate follows from a similar way. Thus, the proof is completed.
\end{proof}

\section{Numerical experiments}\label{s3.5}
In this section we will present three numerical experiments to illustrate the theoretical analysis.

\begin{example}\label{ex3.5.1}
Consider a  mean field SDE driven by multiplicative noise
\begin{equation}\label{e3.5.1}
\begin{aligned}
dX(t)&=0.1(X(t)+\sin{t})dt+0.1\int_{\mathbb{R}^d}\frac{\sin y}{1+y^2}\mu_t(dy)dt+\frac{1}{\sqrt{10}}X(t) dW(t),\\
X(0)&=X_0\sim N(0,1),
\end{aligned}
\end{equation}
where $\mu_t={\mathcal{L}}_{X(t)}$.
\end{example}

The corresponding Fokker-Planck equation reads
\begin{equation}\label{e3.5.2}
\begin{aligned}
\frac{\partial p(t,x)}{\partial t}&=0.1\frac{\partial [(x+\sin{t}+\int_{\mathbb{R}}\frac{\sin y}{1+y^2}p(t,y)dy)p(t,x)]}{\partial x}+\frac15\frac{\partial^2 [x^2p(t,x)]}{\partial x^2}
\end{aligned}
\end{equation}
and  $p_0(x)=\frac{1}{\sqrt{2\pi}}e^{-x^2}$.

Take $[0,T]=[0,1]$ and $\mathcal{D}=[-6,6]$. We choose $\kappa_0=2^{-14}$ and $h_0=12\cdot2^{-10}$  to calculate an approximate density function $p_{\mathcal D,0}^{n,k}$  which is  regarded as a referee exact solution to \eqref{e3.5.2}. Then, we  select different time step sizes $\kappa=2^{-9}, 2^{-10}, 2^{-11}, 2^{-12}$ to approximate \eqref{e3.5.2}, to check the corresponding temporal error estimates and convergence orders, which are presented in 
Table.\ \ref{table3.5.1} and Fig.\ \ref{f3.1}, respectively. Meanwhile, we  take $h/12=2^{-8},2^{-7},2^{-6},2^{-5}$
to carry out numerical computation and detect spatial error estimates and convergence orders, see Table.\ \ref{tab3.2} and Fig.\ \ref{f3.2}, respectively.

Denote by $p_{\kappa_0,h_0}(t,x)$ a continuous version of $p_{\mathcal D,0}^{n,k}$ as defined at the beginning of Section \ref{s3.4.1}. We simulate $10^5$ many sample trajectories using Euler-Maruyama scheme \eqref{e3.4.4} for each step size $\kappa=2^{-9}, 2^{-10}, 2^{-11}, 2^{-12}, 2^{-14}$ to verify the error estimates and convergence orders, as shown in  Table.\  \ref{tab3.3} and Fig.\ \ref{f3.3}.

The above results imply that the numerical computations for approximating PDE and SDE are effective and accurate. In order to show the validity of our method, we simulate the expectation and variance of the solution $X(t)$ to \eqref{e3.5.1} using three different methods. First, we  directly calculate the expectation and variance by virtue of the numerical density function, which are listed in the first line in Table.\ \ref{tab3.4}. Then, we use the computed  $10^4$ sample trajectories to simulate the expectation and variance, see the second line in Table.\ \ref{tab3.4}. Finally, we apply particle method with $10^3$ particles and $10^4$ sample trajectories to approximate \eqref{e3.5.1} and calculate expectation and variance, cf.\ the third line in Table.\ \ref{tab3.4}.

\begin{table*}[h]
   \centering
				\begin{tabular}{@{}llll@{}} 
                  \toprule
					$\kappa$  & error &  order \\
                    \midrule
                   $2^{-12}$&4.3019e-06& -  \\
                     $2^{-11}$&1.0036e-05&1.2221 \\
                     $2^{-10}$&2.1495e-05&1.0989\\
                     $2^{-9}$&4.4383e-05&1.0460 \\
                     \botrule
				\end{tabular}
			\caption{Error and  convergence order.}
				\label{table3.5.1}
			\centering
				\begin{tabular}{@{}llll@{}}
             \toprule
					$h$  & error &  order \\
                    \midrule
                    $0.0469$&     2.1936e-05  & -  \\
					$0.0938$&      9.2028e-05  &  2.0688 \\
					$0.1875$&      3.7104e-04 &  2.0114 \\					
                    $0.3750$&      0.0015 &  1.9833 \\
                     \botrule
				\end{tabular}
			\caption{Error and convergence order.}
				\label{tab3.2}		
\end{table*}

\begin{figure}[H]
\parbox{6cm}{
    \centering
    \includegraphics[width=5cm]{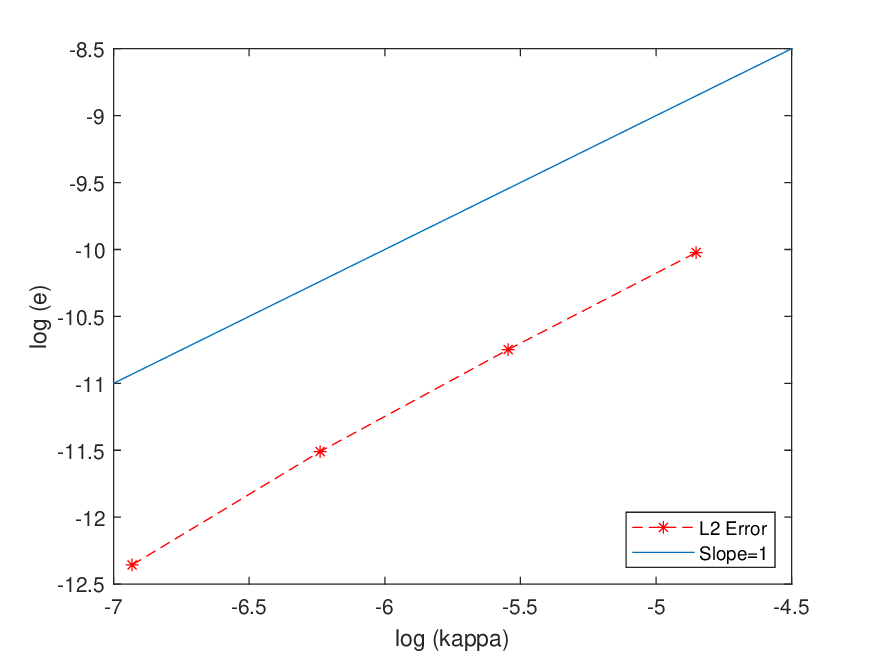}
    \caption{ $\log-\log$ error.}
    \label{f3.1}
    }
\parbox{6cm}{
    \centering
    \includegraphics[width=5cm]{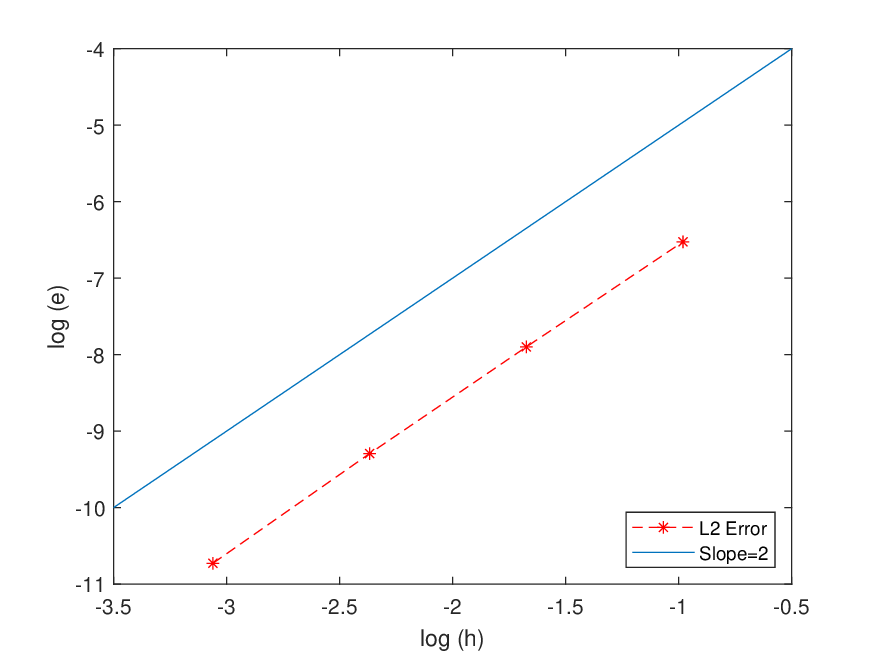}
    \caption{$\log-\log$ error.}
    \label{f3.2}
    }
\end{figure}

\begin{figure}[H]
\parbox{6cm}{
    \centering
    \begin{tabular}{ccc}\hline
                      $\kappa$  & error &  order \\
                    \hline
                    $2^{-12}$&       0.0081  & -  \\
					$2^{-11}$&      0.0128  &  0.5488  \\
					$2^{-10}$&     0.0187 &  0.5683   \\					
                    $2^{-9}$&        0.0277 & 0.5646 \\
                     \hline
\end{tabular}
    \captionof{table}{Error and convergence order.}
    \label{tab3.3}
    }
    \parbox{6cm}{
    \centering
    \includegraphics[width=5cm]{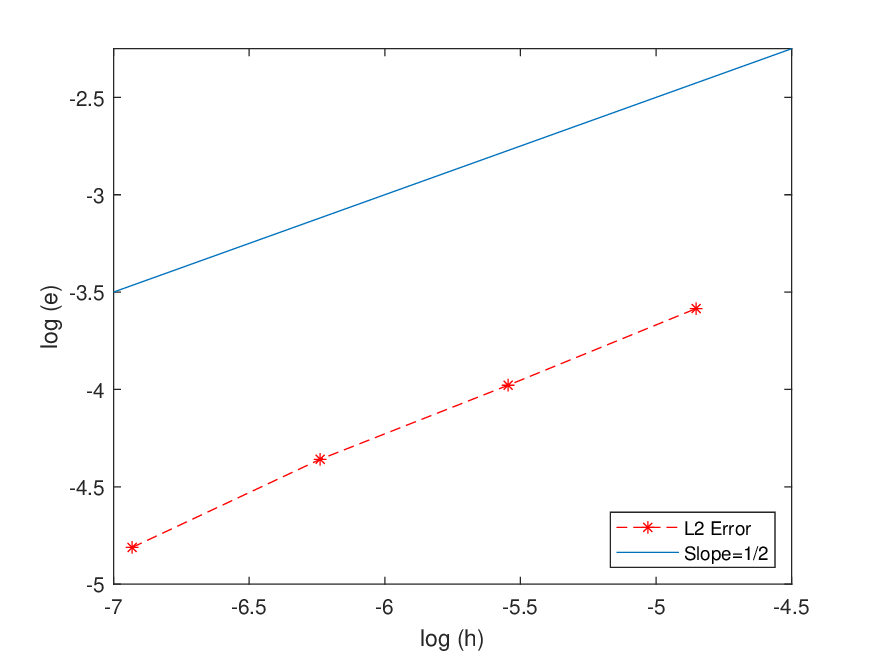}
    \caption{$\log-\log$ error.}
    \label{f3.3}
    }
\end{figure}

\begin{figure}[H]
\centering
\begin{tabular}{ccc}
\hline
                        & $E_X$&  $V_X$ \\
                    \hline
                   PDF method&0.0476& 1.3305  \\
                      Sample trajectory method&0.0474 & 1.3476\\
                    Particle method  & 0.0462 & 1.3381 \\
                     \hline
\end{tabular}
\captionof{table}{Expectation and variance.}
\label{tab3.4}
\end{figure}

\begin{example}
Consider a $2$-dimensional mean field SDE driven by additive noise
\begin{equation}\label{e3.5.5}
\begin{aligned}
dX(t)=&0.1\begin{pmatrix}
\sqrt{(X_1(t))^2+(X_2(t))^2+0.4}\\
\sqrt{(X_1(t))^2+(X_2(t))^2+0.4}
\end{pmatrix}dt+
0.1\begin{pmatrix}
\int_{\mathbb{R}^2}\frac{\sin y_1}{1+y_1^2}\frac{\sin y_2}{1+y_2^2}\mu_t(dy)\\
\int_{\mathbb{R}^2}\frac{\sin y_1}{1+y_1^2}\frac{\sin y_2}{1+y_2^2}\mu_t(dy)
\end{pmatrix}dt\\
&\!+\!\begin{pmatrix}
\frac{2}{\sqrt10}dW_1(t)+\frac{1}{\sqrt10}dW_2(t)\\
\frac{1}{\sqrt10}dW_1(t)+\frac{2}{\sqrt10}dW_2(t)
\end{pmatrix},\\
X(0)=&(X_0^1,X_0^2)\overset{iid}{\sim}N(0,0.04).
\end{aligned}
\end{equation}
\end{example}
The corresponding Fokker-Planck equation reads
\begin{equation}\label{e3.5.6}
\begin{aligned}
\frac{\partial p}{\partial t}\!=\!&
-0.1\frac{\partial [(\sqrt{x_1^2+x_2^2+0.4})+\int_{\mathbb{R}^2}\frac{\sin y_1}{1+y_1^2}\frac{\sin y_2}{1+y_2^2}p(t,y_1,y_2)dy_1dy_2p]
}{\partial x_1}+\frac14\frac{\partial^2p}{\partial x_1^2}\\
&\!-\!0.1\frac{\partial[\sqrt{x_1^2\!+\!x_2^2\!+\!0.4}\!+\!\int_{\mathbb{R}^2}\frac{\sin y_1}{1+y_1^2}\frac{\sin y_2}{1+y_2^2}p(t,y_1,y_2)dy_1dy_2)p]
}{\partial x_2}\!+\!\frac14\frac{\partial^2p}{\partial x_2^2}
\!+\!\frac25\frac{\partial^2p}{\partial x_1\partial x_2}
\end{aligned}
\end{equation}
and $p_0(x_1,x_2)=\frac{1}{0.08\pi}e^{-\frac12(\frac{x_1^2}{0.04}+\frac{x_2^2}{0.04})}$.

Choose  ${\mathcal D} = (-1,1)^2$,  fix $ \kappa_0 = 2^{-5} $ and $ h_0 = 2^{-5}$ to compute a referee exact solution to \eqref{e3.5.6}. Then, we examine the corresponding  temporal and spatial error estimates  and convergence orders and the results are shown in Tables.\ \ref{tab3.3.5}-\ref{tab3.3.6} and Figs.\ \ref{f3.3.3}-\ref{f3.3.4}, respectively.

Since the domain $\mathcal D$ influences the approximation accuracy and  in order to well approximate the density function, we reset ${\mathcal D} = (-4,4)^2$,  $\kappa_0 =2^{-8}$ and $h_0 = 2^{-3}$ to solve \eqref{e3.5.6}. Then, we scan the errors and convergence orders for Euler-Maruyama method with $10^4$ sample trajectories, and results are exhibited  in Table.\ \ref{tab3.3.7} and Fig.\ \ref{f3.3.7}. Finally, we compare the expectation and covariance obtained by three different methods in Table.\ \ref{tab3.3.4}.

\begin{table*}[h]
   \centering
		\begin{tabular}{@{}llll@{}}
		          \toprule
					 $\kappa$  & error &  order \\
                    \midrule
                    $2^{-10}$&     6.8906e-04  & -  \\
					$2^{-9}$&      0.0016  &  1.2246 \\
					$2^{-8}$&      0.0035 &  1.1039  \\					
                    $2^{-7}$&      0.0072 &  1.0561 \\
                     \botrule
				\end{tabular}
			\caption{Error and  convergence order.}
				 \label{tab3.3.5}
			\centering
				\begin{tabular}{@{}llll@{}}
		         \toprule
					  $h$  & error &  order \\
                    \midrule
                   $2^{-4}$&0.0031& -  \\
                     $2^{-3}$&0.0149&2.2582 \\
                     $2^{-2}$&0.0564&1.9209\\
                     \botrule
				\end{tabular}
			\caption{Error and convergence order.}
				\label{tab3.3.6}		
\end{table*}

\begin{figure}[H]
\parbox{6cm}{
    \centering
    \includegraphics[width=5cm]{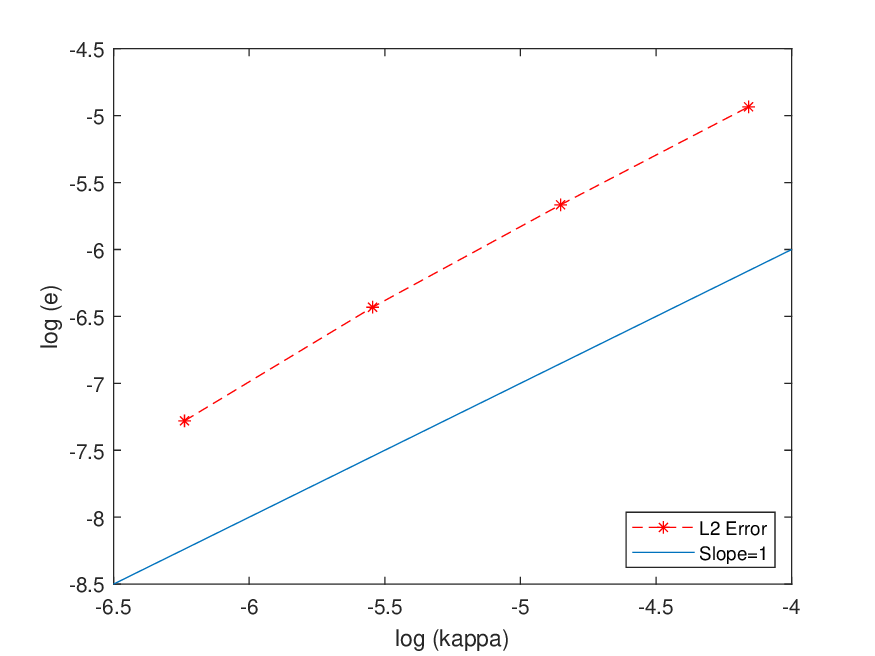}
    \caption{$\log-\log$ error.}
    \label{f3.3.3}
    }
\parbox{6cm}{
    \centering
    \includegraphics[width=5cm]{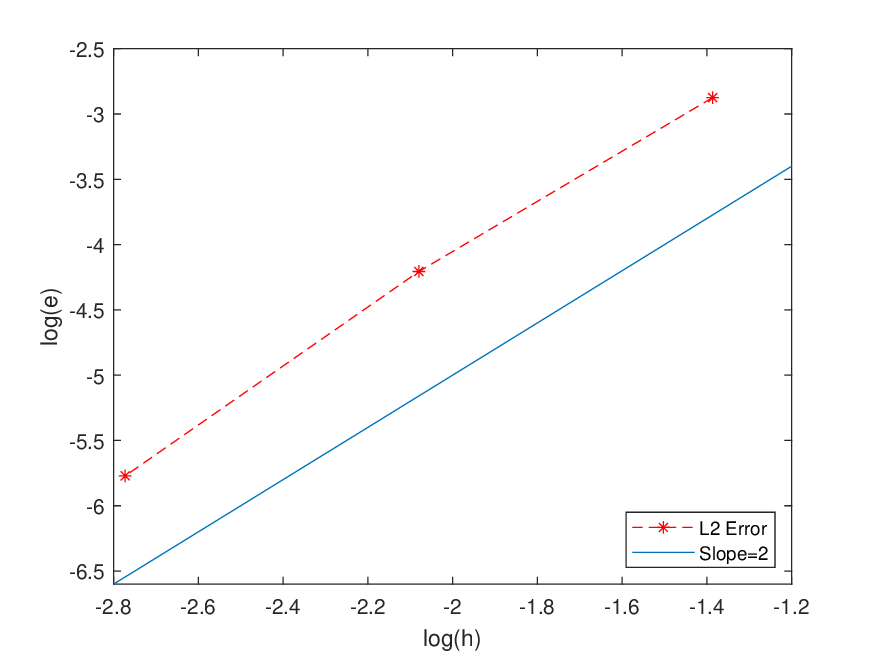}
    \caption{$\log-\log$ error.}
    \label{f3.3.4}
    }
\end{figure}

\begin{figure}[H]
\parbox{6cm}{
    \centering
    \begin{tabular}{ccc}\hline
                      $\kappa$  & error &  order \\
                    \hline
                    $2^{-10}$&       1.1520e-04  & -  \\
					$2^{-9}$&      2.6185e-04  &  1.1846  \\
					$2^{-8}$&     5.5042e-04 &  1.0718   \\					
                    $2^{-7}$&        0.0011 & 1.0461 \\
                     \hline
\end{tabular}
    \captionof{table}{Error and convergence order.}
    \label{tab3.3.7}
    }
    \parbox{6cm}{
    \centering
    \includegraphics[width=5cm]{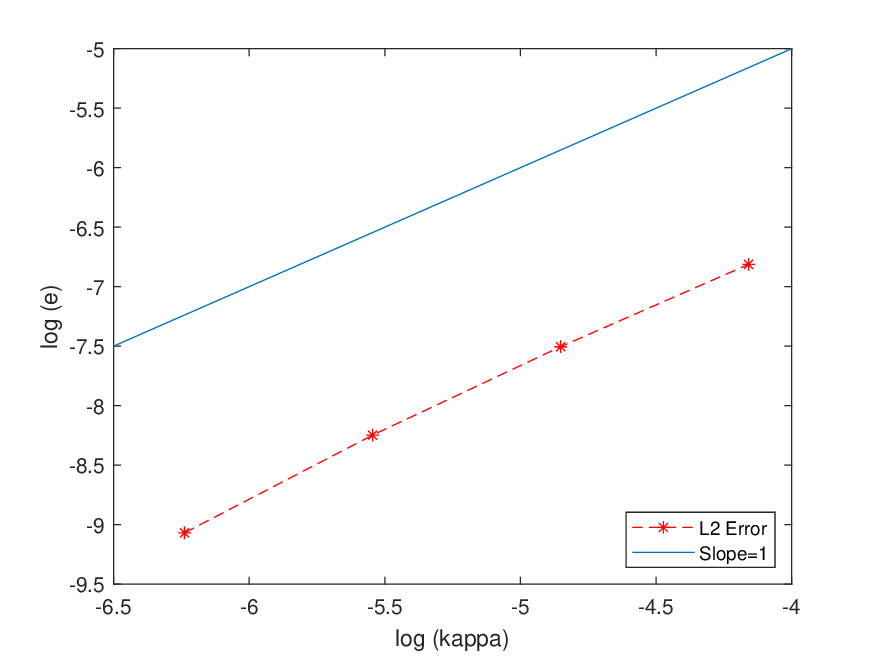}
    \caption{$\log-\log$ error.}
    \label{f3.3.7}
    }
\end{figure}


\begin{figure}[H]
\centering
\begin{tabular}{ccccccc}\hline
                        & $E_{X_1}$& $E_{X_2}$ & $V_{X_1}$& $V_{X_2}$& $V_{X_1X_2}$ \\
                    \hline
                   PDF method&0.0956& 0.0956 &0.5447& 0.5447&0.4047\\
                      Sample trajectories method&0.0945& 0.0957& 0.5499&  0.5470& 0.4079\\
                    Particle method&0.0937 &0.0925&0.5458& 0.5375 &0.4032  \\
                     \hline
\end{tabular}
\captionof{table}{Expectation and covariance.}
\label{tab3.3.4}
\end{figure}

\begin{example}
Consider a $2$-dimensional mean-field SDE driven by additive noise
\begin{equation}\label{e3.5.4}
\begin{aligned}
dX(t)=&0.1\begin{pmatrix}
-\frac{3}{2}X_1(t)+\frac{1}{2}X_2(t)+\sin(2\pi t)\\
\frac{1}{3}X_1(t)-\frac{4}{3}X_2(t)+\cos(2\pi t)
\end{pmatrix}dt\\
&+
0.1\begin{pmatrix}
\int_{\mathbb{R}^2}\frac{\sin y_1}{1+y_1^2}\frac{\sin y_2}{1+y_2^2}\mu_t(dy)\\
\int_{\mathbb{R}^2}\frac{\sin y_1}{1+y_1^2}\frac{\sin y_2}{1+y_2^2}\mu_t(dy)
\end{pmatrix}dt
+\begin{pmatrix}
\frac{1}{10}\\
\frac{1}{10}
\end{pmatrix}dW(t),\\
X(0)=&(X_0^1,X_0^2)\overset{iid}{\sim}N(0,0.01),
\end{aligned}
\end{equation}
where $W(t)$ is a $1$-dimensional Brownian motion.
\end{example}
The corresponding Fokker-Planck equation reads
\begin{equation}\label{e3.5.3}
\begin{aligned}
\frac{\partial p}{\partial t}=&
-0.1\frac{\partial [(-\frac{3}{2}x_1(t)+\frac{1}{2}x_2(t)+\sin(2\pi t)+\int_{\mathbb{R}^2}\frac{\sin y_1}{1+y_1^2}\frac{\sin y_2}{1+y_2^2}p(t,y_1,y_2)dy_1dy_2)p]}{\partial x_1}\\
&-0.1\frac{\partial[(\frac{1}{3}x_1(t)-\frac{4}{3}x_2(t)+\cos(2\pi t)+\int_{\mathbb{R}^2}\frac{\sin y_1}{1+y_1^2}\frac{\sin y_1}{1+y_1^2}p(t,y_1,y_2)dy_1dy_2)p]
}{\partial x_2}\\
&+\frac{1}{200}\frac{\partial^2p}{\partial x_1^2}+\frac{1}{100}\frac{\partial^2p}{\partial x_1\partial x_2}+\frac{1}{200}\frac{\partial^2p}{\partial x_2^2}
\end{aligned}
\end{equation}
and $p_0(x_1,x_2)=\frac{1}{0.02\pi}e^{-\frac12(\frac{x_1^2}{0.01}+\frac{x_2^2}{0.01})}$.

Notice that $\sigma=(\frac{1}{10},\frac{1}{10})^T$ and the rank of  matrix $A=\begin{pmatrix}
  \frac{1}{100}&\frac{1}{100} \\
  \frac{1}{100}&\frac{1}{100}
\end{pmatrix}$ is of  $1$. Then the assumption (A4) is violated and the Fokker-Planck equation \eqref{e3.5.3} becomes degenerate. In this situation, we can not ensure the well-posedness of \eqref{e3.5.3}. However, our method can be used successfully to approximate \eqref{e3.5.4} and \eqref{e3.5.3}. It is point out in  \cite{JinYonVitXue2025} that the spatial convergence order is of $1$. We choose ${\mathcal D}= [-1,1]^2$ to check the corresponding temporal and spatial error estimates and convergence orders, see Tables.\ \ref{tab3.4.1}-\ref{tab3.4.2} and Figs.\ \ref{f3.5.1}-\ref{f3.5.2}, respectively. Then, we enlarge ${\mathcal D}=(-4,4)^2$ to verify errors and convergence orders of Euler-Maruyama method \eqref{e3.4.4} and expectation and covariance, cf.\ Table.\ \ref{tab3.5.1}.

\begin{table*}[h]
   \centering
				\begin{tabular}{ccc} \hline
				$\kappa$  & error &  order \\
                    \hline
                    $2^{-6}$&     0.0178  & -  \\
					$2^{-5}$&       0.0426  &  1.2552 \\
					$2^{-4}$&      0.0946 &  1.1528  \\					
                    $2^{-3}$&      0.2023 &  1.0960   \\
                     \hline
				\end{tabular}
			\caption{Error and  convergence order.}
				  \label{tab3.4.1}
			\centering
				\begin{tabular}{ccc}\hline
					  $h$  & error &  order \\
                   \hline
                     $\frac{2}{48}$&0.0867& -  \\
                     $\frac{2}{24}$&0.1973&1.1854 \\
                     $\frac{2}{6}$&0.4114&1.0604 \\
                     \hline
				\end{tabular}
			\caption{Error and convergence order.}
				 \label{tab3.4.2}		
\end{table*}

\begin{figure}[H]
\parbox{6cm}{
    \centering
    \includegraphics[width=5cm]{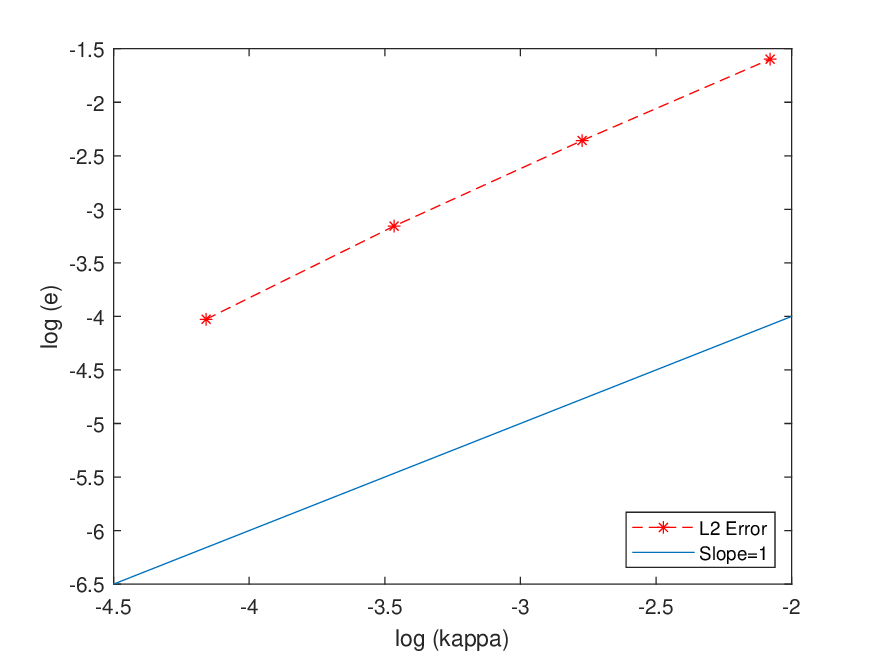}
    \caption{$\log-\log$ error.}
    \label{f3.5.1}
    }
\parbox{6cm}{
    \centering
    \includegraphics[width=5cm]{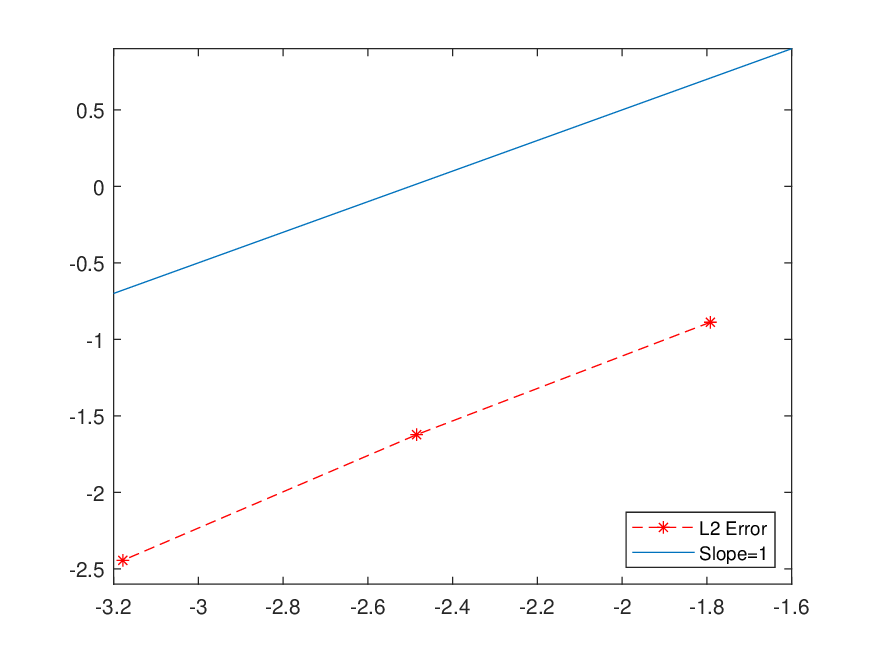}
    \caption{$\log-\log$ error.}
     \label{f3.5.2}
    }
\end{figure}

\begin{figure}[H]
\parbox{6cm}{
    \centering
    \begin{tabular}{ccc}\hline
                      $\kappa$  & error &  order \\
                    \hline
                    $2^{-6}$&       1.2140e-04  & -  \\
					$2^{-5}$&       2.7785e-04 &   1.1984    \\
					$2^{-4}$&     5.9639e-04 &  1.0953  \\					
                    $2^{-3}$&        0.0013 & 1.0785  \\
                     \hline
\end{tabular}
    \captionof{table}{Error and convergence order.}
    }
    \parbox{6cm}{
    \centering
    \includegraphics[width=5cm]{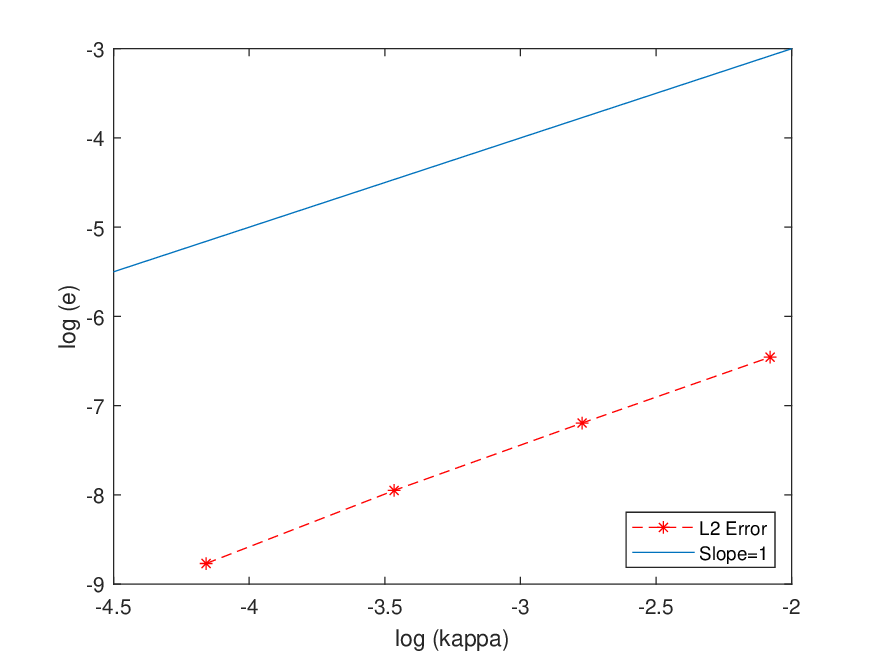}
    \caption{$\log-\log$ error.}
    \label{f3.5.3}
    }
\end{figure}

\begin{figure}[H]
\centering
\begin{tabular}{ccccccc}\hline
                        & $E_{X_1}$& $E_{X_2}$ & $V_{X_1}$& $V_{X_2}$& $V_{X_1X_2}$ \\
                    \hline
                   PDF method& 0& 0  &0.0165 &0.0167  &0.0097 \\
               Sample trajectories method& -1.1977e-04&6.0295e-04 &0.0161 &0.0165 &0.0097 \\
                    Particles method& -9.0079e-04 & 4.9980e-04&0.0169 &0.0165  &  0.0098 \\
                     \hline
\end{tabular}
\captionof{table}{Expectation and covariance.}
\label{tab3.5.1}
\end{figure}

\section*{Acknowledgments}
This research is supported by Jilin Provincial Department of Science and Technology
(20240301017GX) and National Natural Science Foundation of China (12171199).

\bibliographystyle{AIMS}
\bibliography{references}


\begin{thebibliography}{27}
\ifx \bisbn   \undefined \def \bisbn  #1{ISBN #1}\fi
\ifx \binits  \undefined \def \binits#1{#1}\fi
\ifx \bauthor  \undefined \def \bauthor#1{#1}\fi
\ifx \batitle  \undefined \def \batitle#1{#1}\fi
\ifx \bjtitle  \undefined \def \bjtitle#1{#1}\fi
\ifx \bvolume  \undefined \def \bvolume#1{\textbf{#1}}\fi
\ifx \byear  \undefined \def \byear#1{#1}\fi
\ifx \bissue  \undefined \def \bissue#1{#1}\fi
\ifx \bfpage  \undefined \def \bfpage#1{#1}\fi
\ifx \blpage  \undefined \def \blpage #1{#1}\fi
\ifx \burl  \undefined \def \burl#1{\textsf{#1}}\fi
\ifx \doiurl  \undefined \def \doiurl#1{\url{https://doi.org/#1}}\fi
\ifx \betal  \undefined \def \betal{\textit{et al.}}\fi
\ifx \binstitute  \undefined \def \binstitute#1{#1}\fi
\ifx \binstitutionaled  \undefined \def \binstitutionaled#1{#1}\fi
\ifx \bctitle  \undefined \def \bctitle#1{#1}\fi
\ifx \beditor  \undefined \def \beditor#1{#1}\fi
\ifx \bpublisher  \undefined \def \bpublisher#1{#1}\fi
\ifx \bbtitle  \undefined \def \bbtitle#1{#1}\fi
\ifx \bedition  \undefined \def \bedition#1{#1}\fi
\ifx \bseriesno  \undefined \def \bseriesno#1{#1}\fi
\ifx \blocation  \undefined \def \blocation#1{#1}\fi
\ifx \bsertitle  \undefined \def \bsertitle#1{#1}\fi
\ifx \bsnm \undefined \def \bsnm#1{#1}\fi
\ifx \bsuffix \undefined \def \bsuffix#1{#1}\fi
\ifx \bparticle \undefined \def \bparticle#1{#1}\fi
\ifx \barticle \undefined \def \barticle#1{#1}\fi
\bibcommenthead
\ifx \bconfdate \undefined \def \bconfdate #1{#1}\fi
\ifx \botherref \undefined \def \botherref #1{#1}\fi
\ifx \url \undefined \def \url#1{\textsf{#1}}\fi
\ifx \bchapter \undefined \def \bchapter#1{#1}\fi
\ifx \bbook \undefined \def \bbook#1{#1}\fi
\ifx \bcomment \undefined \def \bcomment#1{#1}\fi
\ifx \oauthor \undefined \def \oauthor#1{#1}\fi
\ifx \citeauthoryear \undefined \def \citeauthoryear#1{#1}\fi
\ifx \endbibitem  \undefined \def \endbibitem {}\fi
\ifx \bconflocation  \undefined \def \bconflocation#1{#1}\fi
\ifx \arxivurl  \undefined \def \arxivurl#1{\textsf{#1}}\fi
\csname PreBibitemsHook\endcsname

\bibitem[\protect\citeauthoryear{Black and Scholes}{1973}]{BlaSch1973}
\begin{barticle}
\bauthor{\bsnm{Black}, \binits{F.}},
\bauthor{\bsnm{Scholes}, \binits{M.}}:
\batitle{The pricing of options and corporate liabilities}.
\bjtitle{J. Polit. Econ.}
\bvolume{81}(\bissue{3}),
\bfpage{637}--\blpage{654}
(\byear{1973})
\doiurl{10.1086/260062}
\end{barticle}
\endbibitem

\bibitem[\protect\citeauthoryear{Santill\'an}{2014}]{San2014}
\begin{bbook}
\bauthor{\bsnm{Santill\'an}, \binits{M.}}:
\bbtitle{Chemical Kinetics, Stochastic Processes, and Irreversible
  Thermodynamics}.
\bsertitle{Lecture Notes on Mathematical Modelling in the Life Sciences},
p. \bfpage{126}
(\byear{2014}).
\doiurl{10.1007/978-3-319-06689-9} .
\burl{https://doi.org/10.1007/978-3-319-06689-9}
\end{bbook}
\endbibitem

\bibitem[\protect\citeauthoryear{Dagan}{1982}]{Dag1982}
\begin{barticle}
\bauthor{\bsnm{Dagan}, \binits{G.}}:
\batitle{Stochastic modeling of groundwater flow by unconditional and
  conditional probabilities: 1. conditional simulation and the direct problem}.
\bjtitle{Water Resources Research}
\bvolume{18}(\bissue{4}),
\bfpage{813}--\blpage{833}
(\byear{1982})
\doiurl{10.1029/WR018i004p00813}
{\href{https://arxiv.org/abs/https://agupubs.onlinelibrary.wiley.com/doi/pdf/10.1029/WR018i004p00813}{{https://agupubs.onlinelibrary.wiley.com/doi/pdf/10.1029/WR018i004p00813}}}
\end{barticle}
\endbibitem

\bibitem[\protect\citeauthoryear{Baladron et~al.}{2012}]{Bal2012}
\begin{barticle}
\bauthor{\bsnm{Baladron}, \binits{J.}},
\bauthor{\bsnm{Fasoli}, \binits{D.}},
\bauthor{\bsnm{Faugeras}, \binits{O.D.}},
\bauthor{\bsnm{Touboul}, \binits{J.}}:
\batitle{Mean-field description and propagation of chaos in networks of
  hodgkin-huxley and fitzhugh-nagumo neurons}.
\bjtitle{Journal of Mathematical Neuroscience}
\bvolume{2},
\bfpage{10}--\blpage{10}
(\byear{2012})
\end{barticle}
\endbibitem

\bibitem[\protect\citeauthoryear{Patlak}{1953}]{Cli1953}
\begin{barticle}
\bauthor{\bsnm{Patlak}, \binits{C.S.}}:
\batitle{Random walk with persistence and external bias}.
\bjtitle{Bulletin of Mathematical Biology}
\bvolume{15},
\bfpage{311}--\blpage{338}
(\byear{1953})
\end{barticle}
\endbibitem

\bibitem[\protect\citeauthoryear{Keller and Segel}{1970}]{Eve1970}
\begin{barticle}
\bauthor{\bsnm{Keller}, \binits{E.F.}},
\bauthor{\bsnm{Segel}, \binits{L.A.}}:
\batitle{Initiation of slime mold aggregation viewed as an instability}.
\bjtitle{Journal of Theoretical Biology}
\bvolume{26}(\bissue{3}),
\bfpage{399}--\blpage{415}
(\byear{1970})
\doiurl{10.1016/0022-5193(70)90092-5}
\end{barticle}
\endbibitem

\bibitem[\protect\citeauthoryear{Banner et~al.}{2005}]{BanFerKar2005}
\begin{botherref}
\oauthor{\bsnm{Banner}, \binits{A.D.}},
\oauthor{\bsnm{Fernholz}, \binits{R.}},
\oauthor{\bsnm{Karatzas}, \binits{I.}}:
Atlas models of equity markets.
The Annals of Applied Probability
\textbf{15}(4)
(2005)
\doiurl{10.1214/105051605000000449}
\end{botherref}
\endbibitem

\bibitem[\protect\citeauthoryear{Jourdain and Reygner}{2014}]{BenJul2014}
\begin{botherref}
\oauthor{\bsnm{Jourdain}, \binits{B.}},
\oauthor{\bsnm{Reygner}, \binits{J.}}:
Capital distribution and portfolio performance in the mean-field Atlas model
(2014).
\url{https://arxiv.org/abs/1312.5660}
\end{botherref}
\endbibitem

\bibitem[\protect\citeauthoryear{Bossy and Talay}{1997}]{BosTal1997}
\begin{barticle}
\bauthor{\bsnm{Bossy}, \binits{M.}},
\bauthor{\bsnm{Talay}, \binits{D.}}:
\batitle{A stochastic particle method for the {M}c{K}ean-{V}lasov and the
  {B}urgers equation}.
\bjtitle{Math. Comp.}
\bvolume{66}(\bissue{217}),
\bfpage{157}--\blpage{192}
(\byear{1997})
\doiurl{10.1090/S0025-5718-97-00776-X}
\end{barticle}
\endbibitem

\bibitem[\protect\citeauthoryear{dos Reis et~al.}{2022}]{dosEng2022}
\begin{barticle}
\bauthor{\bsnm{Reis}, \binits{G.c.}},
\bauthor{\bsnm{Engelhardt}, \binits{S.}},
\bauthor{\bsnm{Smith}, \binits{G.}}:
\batitle{Simulation of {M}c{K}ean-{V}lasov {SDE}s with super-linear growth}.
\bjtitle{IMA J. Numer. Anal.}
\bvolume{42}(\bissue{1}),
\bfpage{874}--\blpage{922}
(\byear{2022})
\doiurl{10.1093/imanum/draa099}
\end{barticle}
\endbibitem

\bibitem[\protect\citeauthoryear{Chen and dos Reis}{2022}]{Chedos2022}
\begin{barticle}
\bauthor{\bsnm{Chen}, \binits{X.}},
\bauthor{\bsnm{Reis}, \binits{G.c.}}:
\batitle{A flexible split-step scheme for solving {M}c{K}ean-{V}lasov
  stochastic differential equations}.
\bjtitle{Appl. Math. Comput.}
\bvolume{427},
\bfpage{127180}--\blpage{23}
(\byear{2022})
\doiurl{10.1016/j.amc.2022.127180}
\end{barticle}
\endbibitem

\bibitem[\protect\citeauthoryear{Reisinger and Stockinger}{2022}]{ReiSto2022}
\begin{barticle}
\bauthor{\bsnm{Reisinger}, \binits{C.}},
\bauthor{\bsnm{Stockinger}, \binits{W.}}:
\batitle{An adaptive {E}uler-{M}aruyama scheme for {M}c{K}ean-{V}lasov {SDE}s
  with super-linear growth and application to the mean-field
  {F}itz{H}ugh-{N}agumo model}.
\bjtitle{J. Comput. Appl. Math.}
\bvolume{400},
\bfpage{113725}--\blpage{23}
(\byear{2022})
\doiurl{10.1016/j.cam.2021.113725}
\end{barticle}
\endbibitem

\bibitem[\protect\citeauthoryear{Belomestny and
  Schoenmakers}{2018}]{BelSch2018}
\begin{barticle}
\bauthor{\bsnm{Belomestny}, \binits{D.}},
\bauthor{\bsnm{Schoenmakers}, \binits{J.}}:
\batitle{Projected particle methods for solving {M}c{K}ean-{V}lasov stochastic
  differential equations}.
\bjtitle{SIAM J. Numer. Anal.}
\bvolume{56}(\bissue{6}),
\bfpage{3169}--\blpage{3195}
(\byear{2018})
\doiurl{10.1137/17M1111024}
\end{barticle}
\endbibitem

\bibitem[\protect\citeauthoryear{Liu}{2024}]{Liu2024}
\begin{barticle}
\bauthor{\bsnm{Liu}, \binits{Y.}}:
\batitle{Particle method and quantization-based schemes for the simulation of
  the {M}c{K}ean-{V}lasov equation}.
\bjtitle{ESAIM Math. Model. Numer. Anal.}
\bvolume{58}(\bissue{2}),
\bfpage{571}--\blpage{612}
(\byear{2024})
\doiurl{10.1051/m2an/2024007}
\end{barticle}
\endbibitem

\bibitem[\protect\citeauthoryear{Jin et~al.}{2020}]{JinLiLiu2020}
\begin{barticle}
\bauthor{\bsnm{Jin}, \binits{S.}},
\bauthor{\bsnm{Li}, \binits{L.}},
\bauthor{\bsnm{Liu}, \binits{J.-G.}}:
\batitle{Random batch methods ({RBM}) for interacting particle systems}.
\bjtitle{J. Comput. Phys.}
\bvolume{400},
\bfpage{108877}--\blpage{30}
(\byear{2020})
\doiurl{10.1016/j.jcp.2019.108877}
\end{barticle}
\endbibitem

\bibitem[\protect\citeauthoryear{Jin et~al.}{2021}]{JiLiLiu2021}
\begin{barticle}
\bauthor{\bsnm{Jin}, \binits{S.}},
\bauthor{\bsnm{Li}, \binits{L.}},
\bauthor{\bsnm{Liu}, \binits{J.-G.}}:
\batitle{Convergence of the random batch method for interacting particles with
  disparate species and weights}.
\bjtitle{SIAM J. Numer. Anal.}
\bvolume{59}(\bissue{2}),
\bfpage{746}--\blpage{768}
(\byear{2021})
\doiurl{10.1137/20M1327641}
\end{barticle}
\endbibitem

\bibitem[\protect\citeauthoryear{Jin and Li}{2022}]{JinLi2022}
\begin{barticle}
\bauthor{\bsnm{Jin}, \binits{S.}},
\bauthor{\bsnm{Li}, \binits{L.}}:
\batitle{On the mean field limit of the random batch method for interacting
  particle systems}.
\bjtitle{Sci. China Math.}
\bvolume{65}(\bissue{1}),
\bfpage{169}--\blpage{202}
(\byear{2022})
\doiurl{10.1007/s11425-020-1810-6}
\end{barticle}
\endbibitem

\bibitem[\protect\citeauthoryear{Jin and Li}{[2022] \copyright
  2022}]{JinLi[2022]}
\begin{bchapter}
\bauthor{\bsnm{Jin}, \binits{S.}},
\bauthor{\bsnm{Li}, \binits{L.}}:
\bctitle{Random batch methods for classical and quantum interacting particle
  systems and statistical samplings}.
In: \bbtitle{Active Particles. {V}ol. 3. {A}dvances in Theory, Models, and
  Applications}.
\bsertitle{Model. Simul. Sci. Eng. Technol.},
pp. \bfpage{153}--\blpage{200}
(\byear{[2022] \copyright 2022})
\end{bchapter}
\endbibitem

\bibitem[\protect\citeauthoryear{Evans}{2013}]{Eva2013}
\begin{bbook}
\bauthor{\bsnm{Evans}, \binits{L.C.}}:
\bbtitle{An Introduction to Stochastic Differential Equations},
p. \bfpage{151}
(\byear{2013}).
\doiurl{10.1090/mbk/082} .
\burl{https://doi.org/10.1090/mbk/082}
\end{bbook}
\endbibitem

\bibitem[\protect\citeauthoryear{Kumar et~al.}{2022}]{KumNeeReiSto2022}
\begin{barticle}
\bauthor{\bsnm{Kumar}, \binits{C.}},
\bauthor{\bsnm{Neelima}},
\bauthor{\bsnm{Reisinger}, \binits{C.}},
\bauthor{\bsnm{Stockinger}, \binits{W.}}:
\batitle{Well-posedness and tamed schemes for {M}c{K}ean-{V}lasov equations
  with common noise}.
\bjtitle{Ann. Appl. Probab.}
\bvolume{32}(\bissue{5}),
\bfpage{3283}--\blpage{3330}
(\byear{2022})
\doiurl{10.1214/21-aap1760}
\end{barticle}
\endbibitem

\bibitem[\protect\citeauthoryear{Frank}{2005}]{Fra2005}
\begin{bbook}
\bauthor{\bsnm{Frank}, \binits{T.D.}}:
\bbtitle{Nonlinear {F}okker-{P}lanck Equations}.
\bsertitle{Springer Series in Synergetics},
p. \bfpage{404}
(\byear{2005}).
\bcomment{Fundamentals and applications}
\end{bbook}
\endbibitem

\bibitem[\protect\citeauthoryear{Barbu and R\"ockner}{2020}]{BarRoc2020}
\begin{barticle}
\bauthor{\bsnm{Barbu}, \binits{V.}},
\bauthor{\bsnm{R\"ockner}, \binits{M.}}:
\batitle{From nonlinear {F}okker-{P}lanck equations to solutions of
  distribution dependent {SDE}}.
\bjtitle{Ann. Probab.}
\bvolume{48}(\bissue{4}),
\bfpage{1902}--\blpage{1920}
(\byear{2020})
\doiurl{10.1214/19-AOP1410}
\end{barticle}
\endbibitem

\bibitem[\protect\citeauthoryear{LeVeque}{2007}]{LeV2007}
\begin{bbook}
\bauthor{\bsnm{LeVeque}, \binits{R.J.}}:
\bbtitle{Finite Difference Methods for Ordinary and Partial Differential
  Equations: Steady-state and Time-dependent Problems},
(\byear{2007})
\end{bbook}
\endbibitem

\bibitem[\protect\citeauthoryear{Kloeden and Platen}{1992}]{KloPla1992}
\begin{bbook}
\bauthor{\bsnm{Kloeden}, \binits{P.E.}},
\bauthor{\bsnm{Platen}, \binits{E.}}:
\bbtitle{Numerical Solution of Stochastic Differential Equations}.
\bsertitle{Applications of Mathematics (New York)},
vol. \bseriesno{23},
p. \bfpage{632}
(\byear{1992}).
\doiurl{10.1007/978-3-662-12616-5} .
\burl{https://doi.org/10.1007/978-3-662-12616-5}
\end{bbook}
\endbibitem

\bibitem[\protect\citeauthoryear{Sauer}{2013}]{Sau2013}
\begin{barticle}
\bauthor{\bsnm{Sauer}, \binits{T.}}:
\batitle{Computational solution of stochastic differential equations}.
\bjtitle{WIREs Comput. Stat.}
\bvolume{5}(\bissue{5}),
\bfpage{362}--\blpage{371}
(\byear{2013})
\end{barticle}
\endbibitem

\bibitem[\protect\citeauthoryear{Lord et~al.}{2014}]{LordPoweShard2014}
\begin{bbook}
\bauthor{\bsnm{Lord}, \binits{G.J.}},
\bauthor{\bsnm{Powell}, \binits{C.E.}},
\bauthor{\bsnm{Shardlow}, \binits{T.}}:
\bbtitle{An Introduction to Computational Stochastic {PDE}s}.
\bsertitle{Cambridge Texts in Applied Mathematics},
p. \bfpage{503}
(\byear{2014}).
\doiurl{10.1017/CBO9781139017329} .
\burl{https://doi.org/10.1017/CBO9781139017329}
\end{bbook}
\endbibitem

\bibitem[\protect\citeauthoryear{Zhou et~al.}{2025}]{JinYonVitXue2025}
\begin{botherref}
\oauthor{\bsnm{Zhou}, \binits{J.}},
\oauthor{\bsnm{Zou}, \binits{Y.}},
\oauthor{\bsnm{Konarovskyi}, \binits{V.}},
\oauthor{\bsnm{Yang}, \binits{X.}}:
Numerical approximation and dynamics of periodic solution in distribution of
  stochastic differential equations.
Numerical Algorithms
(2025)
\end{botherref}
\endbibitem

\end{thebibliography}
\end{document}